\documentclass[11pt]{article}

\usepackage{amsmath,amssymb,amscd,amsthm,amsfonts}
\usepackage{graphicx,subfigure}
\usepackage{hyperref}
\usepackage{dsfont}
\usepackage{color}

\newtheorem{theorem}{Theorem}[section]

\newtheorem{lemma}[theorem]{Lemma}
\newtheorem{claim}[theorem]{Claim}

\newtheorem{conjecture}[theorem]{Conjecture}

\newcommand{\rr}{\mathds{R}}

\newcommand{\ff}{\mathcal{F}}

\def\rr{\mathds{R}}
\DeclareMathOperator{\conv}{conv}

\DeclareMathOperator{\col}{col}
\title{Tverberg partitions as weak epsilon-nets}

\author{Pablo Sober\'on}

\begin{document}

\maketitle

\begin{abstract}
We prove a Tverberg-type theorem using the probabilistic method.  Given $\varepsilon >0$, we find the smallest number of partitions of a set $X$ in $R^d$ into $r$ parts needed in order to induce at least one Tverberg partition on every subset of $X$ with at least $\varepsilon |X|$ elements.  This generalizes known results about Tverberg's theorem with tolerance.
\end{abstract}



\section{Introduction}

Tverberg's theorem and the weak $\varepsilon$-net theorem for convex sets are central results describing the combinatorial properties of convex sets.  Their statements are the following

\begin{theorem}[Tverberg 1966, \cite{Tverberg:1966tb}]
Let $r,d$ be positive integers.  Given a set $X$ of $(r-1)(d+1)+1$ points in $\rr^d$, there is a partition of $X$ into $r$ sets whose convex hulls intersect.
\end{theorem}

We call a partition into $r$ sets as above a \textit{Tverberg partition}.  For a set $Y \subset \rr^d$, we denote by $\conv Y$ its convex hull.

\begin{theorem}[Weak $\varepsilon$-net; Alon, B\'ar\'any, F\"uredi, Kleitman 1992, \cite{Alon:1992ek}]
Let $d$ be a positive integer and $\varepsilon >0$.  Then, there is an integer $n = n(\varepsilon, d)$ such that the following holds.  For any finite set $X$ of points in $\rr^d$, there is a set $K \subset \rr^d$ of $n(\varepsilon, d)$ points such that for all $Y \subset X$ with $|Y| \ge \varepsilon |X|$, we have that $\conv Y$ intersects $K$.
\end{theorem}

For an overview of both theorems and how they have shaped discrete geomety, consult \cite{Matousek:2002td, barany2017survey}.  One key aspect of the weak $\varepsilon$-net theorem is that $n(\varepsilon, d)$ does not depend on $|X|$.  The two theorems are closely related to each other.  Tverberg's theorem is an important tool in the proof of the ``first selection lemma'' \cite{Barany:1982va}, which in turn is used to prove the weak $\varepsilon$-net theorem.  Finding upper and lower bounds for $n(\varepsilon, d)$ is a difficult problem.  As an upper bound, for any fixed $d$ we have $n(\varepsilon, d) = O(\varepsilon^{-d}\operatorname{polylog}(\varepsilon^{-1}))$ \cite{CEGGSW95, MW04}.  There are lower bounds superlinear in $1/\varepsilon$, for any fixed $d$ we have $n(\varepsilon,d) = \Omega ((1/\varepsilon)\ln^{d-1}(1/\varepsilon))$ \cite{Bukh:2011vs}.

 The purpose of this paper is to provide a different link between these two theorems.  Just as the weak $\varepsilon$-net gives you a fixed-size set which intersects the convex hull of each not too small subset of $X$, now we seek a fixed number of partitions of $X$, such that for every not too small subset $Y \subset X$, at least one of the partitions induces a Tverberg partition on $Y$.  Unlike the weak $\varepsilon$-net problem, we get an exact value for the number of partitions needed.

Given a partition $\mathcal{P}$ of $X$ and $Y \subset X$, we denote by $\mathcal{P} (Y)$ the restriction of $\mathcal{P}$ on $Y$,
\[
\mathcal{P}(Y) = \{K \cap Y: K \in \mathcal{P}\}.
\]

If $\mathcal{P}$ is a partition into $r$ sets, then $\mathcal{P}(Y)$ is also a partition into $r$ sets, though some may be empty.  With this notation, we can state the main result of this paper.

\begin{theorem}\label{theorem-main}
Let $1\ge\varepsilon > 0$ be a real number and $r,d$ be positive integers.  Then, there is an integer $m=m(\varepsilon, r)$ such that the following is true.  For every sufficiently large finite set $X \subset \rr^d$, there are $m$ partitions $\mathcal{P}_1, \ldots, \mathcal{P}_m$ of $X$ into $r$ parts each such that, for every subset $Y \subset X$ with $|Y| \ge \varepsilon |X|$, there is a $k$ such that $\mathcal{P}_k(Y)$ is a Tverberg partition.  Moreover, we have 
\[
m(\varepsilon,r) = \left\lfloor \frac{\ln\left(\frac{1}{\varepsilon}\right)}{\ln\left(\frac{r}{r-1}\right)}\right\rfloor + 1.
\]
\end{theorem}

An equivalent statement is that $\varepsilon > ((r-1)/r)^m$ if and only if $m(\varepsilon, r) \le m$.  One should notice that $1/\ln(r/(r-1)) \sim r$, so $m(\varepsilon, r) \sim r \ln (1/\varepsilon)$.  One surprising aspect of this result is that $m$ does not depend on the dimension.  The effect of the dimension only appears when we look at how large $X$ must be for the theorem to kick in.  The value for $|X|$ where the theorem starts working is, up to polylogarithmic terms, $m  d r^3 (\varepsilon - ((r-1)/r)^m)^{-2}$.

The proof of Theorem \ref{theorem-main} follows from a repeated application of the probabilistic method, contained in section \ref{section-proof}.  We build up on the techniques of \cite{soberon2016robust} to prove Tverberg-type results by making random partitions.  The key new observation is that, given $m$ partitions of $X$, the number of containment-maximal subsets $Y$ such that $\mathcal{P}_k(Y)$ is not a Tverberg partition for any $k$ is polynomial in $|X|$.

This result is also closely related to Tverberg's theorem with tolerance.

\begin{theorem}[Tverberg with tolerance; Garc\'ia-Col\'in, Raggi, Rold\'an-Pensado 2017, \cite{GRR17tolerant}]\label{theorem-tolerance}
Let $r,t,d$ be positive integers, were $r,d$ are fixed.  There is an integer $N(r,t,d)=rt+o(t)$ such that the following holds.  For any set $X$ of $N$ points in $\rr^d$, there is a partition of $X$ into $r$ sets $X_1, \ldots, X_r$ such that, for all $C \subset X$ of cardinality $t$, we have
\[
\bigcap_{j=1}^r \conv (X_j \setminus C) \neq \emptyset.
\]
\end{theorem}

This is a result that is motivated by earlier work of Larman \cite{Larman:1972tn}, who studied the case $t=1, r=2$.  Theorem \ref{theorem-tolerance} determines the correct leading term as $t$ becomes large.  This result been improved to $N = rt + \tilde{O}(\sqrt{t})$, where the $\tilde{O}$ term hides polylogarithmic factors, and is polynomial in $r,d$ \cite{soberon2016robust}.  In the notation of Theorem \ref{theorem-main}, Theorem \ref{theorem-tolerance} says that if $\varepsilon > 1 - 1/r$, then $m(\varepsilon, r) = 1$.  Improved bounds for small values of $t$ can be found in \cite{Soberon:2012er, MZ14tolerant}.

As the driving engine in the proof of Theorem \ref{theorem-main} is Sarkaria's tensoring technique, described in section \ref{section-preliminaries}, it can be easily modified to get similar versions of a multitude of variations of Tverberg's theorem.  This includes Tverberg ``plus minus'' \cite{barany2016tverberg}, colorful Tverberg with equal coefficients \cite{soberon2015equal} and asymptotic variations of Reay's conjecture \cite{soberon2016robust}.  We do not include those variations explicitly.  We do include an $\varepsilon$-version for the colorful Tverberg theorem in section \ref{section-colored}, as it is closely related to a conjecture in \cite{soberon2016robust}.  

A natural question that follows the results of this paper is to determine whether a topological version of Theorem \ref{theorem-main} also holds.

\section{Preliminaries}\label{section-preliminaries}

\subsection{Sarkaria's technique.}
We start discussing the preliminaries for the the proof of Theorem \ref{theorem-main}.  At the core of the proof is Sarkaria's technique to prove Tverberg's theorem via tensor products \cite{Sarkaria:1992vt, BaranyOnn}.

The goal is to reduce Tverberg's theorem to the colorful Carath\'eodory theorem.

\begin{theorem}[Colorful Carath\'eodory; B\'ar\'any 1982 \cite{Barany:1982va}]\label{theorem-colorcarath}
Let $F_1, \ldots, F_{n+1}$ be sets of points in $\rr^n$.  If $0 \in \conv(F_i)$ for all $i=1, \ldots, n+1$, then we can choose points $x_1 \in F_1, \ldots, x_{n+1} \in F_{n+1}$ so that $0 \in \conv \{x_1, \ldots, x_{n+1}\}$.
\end{theorem}

The set $\{x_1, \ldots, x_{n+1}\}$ is called a \textit{transversal} of $\mathcal{F}=\{F_1, \ldots, F_{n+1}\}$.  Each set $F_i$ is called a \textit{color class}.  For the sake of brevity we do not reproduce Sarkaria's proof, but point out the main ingredients.  We distinguish between Tverberg-type results and colorful Carath\'eodory-type results by denoting the dimension of their ambient spaces by $d$ and $n$, respectively.

Let $X = \{x_1, \ldots, x_N\}$ be a set of points in $\rr^d$ and $r$ a positive integer.  We define $n = (d+1)(r-1)$.  Let $v_1, \ldots, v_r$ be the vertices of a regular simplex in $\rr^{r-1}$ centered at the origin.  We construct the points
\[
\bar{x}_{i,j} = (x_i,1) \otimes v_j \in \rr^{(d+1)(r-1)}= \rr^n,
\]
where $\otimes$ denotes the standard tensor product.  Given two vectors ${v_1} \in \rr^{d_1}, {v_2}\in \rr^{d_2}$, their tensor product ${v_1} \otimes {v_2}$ is simply the $d_1 \times d_2$ matrix $v_1v_2^{T}$ interpreted as a $d_1d_2$-dimensional vector.   These tensor products carry all the information about Tverberg partitions into $r$ parts.

\begin{lemma}\label{lemma-magic}
Let $X = \{x_1, \ldots, x_N\}$ be a finite set of points in $\rr^d$, $r$ be a positive integer.  Then, a partition $X_1, \ldots, X_r$ of $X$ is a Tverberg partition if and only if
\[
{0} \in \conv \{\bar{x}_{i,j} : i,j \ \mbox{are such that } x_i \in X_j\}
\]
\end{lemma}

A lucid explanation of the lemma above can be found in \cite{baranytensors}.  Lemma \ref{lemma-magic} implies that, given $X$, if we consider the sets
\[
F_i =\{\bar{x}_{i,j} : j=1,\ldots, r\} \qquad i=1,\ldots,N,
\]
then finding a Tverberg partition of $X$ into $r$ parts corresponds to finding a transversal of $\mathcal{F}=\{F_1, \ldots, F_N\}$ whose convex hull contains the origin in $\rr^n$.   Since $0 \in \conv F_i$ for each $i$, Theorem \ref{theorem-colorcarath} or a variation can be applied.  Then, by Lemma \ref{lemma-magic}, we obtain a Tverberg partition.

For transversals, there is also a natural notion of restriction.  Given a family $\mathcal{F}$ of sets in $\rr^n$, $\mathcal{G} \subset \mathcal{F}$, and $P$ a transversal of $\mathcal{F}$, we define
\[
P(\mathcal{G})=\{x \in P: x \mbox{ came from a set in } \mathcal{G}\}.
\]

Alternatively, $P(\mathcal{G}) = P \cap (\cup \mathcal{G})$.  In order to prove Theorem \ref{theorem-main}, it is sufficient to prove the following.

\begin{theorem}\label{theorem-epsilon-caratheodory}
Let $r, n$ be positive integers and $1\ge \varepsilon >0$ a real number.  Then, there is an integer $m=m(\varepsilon, r)$ such that the following is true.  For every sufficiently large $N$, if we are given a family $\mathcal{F}$ of $N$ sets in $\rr^n$, such that $0 \in \conv F$ and $|F|=r$ for all $F \in \mathcal{F}$, then there are $m$ transversals $P_1, \ldots, P_m$ of $\mathcal{F}$ with the following property.  For every $\mathcal{G} \subset \mathcal{F}$ with $|\mathcal{G}| \ge \varepsilon|\mathcal{F}|$ there is a $k$ with $0 \in \conv P_k(\mathcal{G})$.

Moreover, we have 
\[
m(\varepsilon,r) = \left\lfloor \frac{\ln\left(\frac{1}{\varepsilon}\right)}{\ln\left(\frac{r}{r-1}\right)}\right\rfloor + 1.
\]
\end{theorem}

Indeed, let us sketch how Theorem \ref{theorem-epsilon-caratheodory} implies Theorem \ref{theorem-main}.  

\begin{proof}
Assume $r,d,\varepsilon,m$ are given, satisfying the last equality of Theorem \ref{theorem-epsilon-caratheodory}. Let $n = (d+1)(r-1)+1$.  Assume that we are given a set $X$ of $N$ points in $\rr^d$, $X = \{x_1, \ldots, x_N\}$, where $N$ is a large positive integer.  For $v_1, \ldots, v_r \in \rr^{r-1}$ as before, we construct the sets
\[
F_i = \{(x_i,1) \otimes v_j : j = 1, \ldots, r\} \subset \rr^n.
\]
Then, we apply Theorem \ref{theorem-epsilon-caratheodory} to the family $\mathcal{F} = \{F_1, \ldots, F_N\}$ and find $m$ transversals $P_1, \ldots, P_m$.  Given a set of indices $I \subset [N]$ such that $|I| \ge \varepsilon N$, consider ${\mathcal{G}}_I = \{F_i : i \in I\}$.  Then, there must be a transversal $P_{i_0}$ such that ${0} \in \conv P_{i_0}({\mathcal{G}}_I)$.  By Lemma \ref{lemma-magic}, this means that the partition $\mathcal{P}_{i_0}$ of $X$ induced by $P_{i_0}$ is a Tverberg partition even when restricted to the set $X_I = \{x_i : i \in X\}$.  In other words, the partitions induced by $P_1, \ldots, P_m$ satisfy the conclusion of Theorem \ref{theorem-main}. 
\end{proof}

We also need the following lemma.  It bounds the complexity of verifying if $0 \in \conv Y$ if $Y \subset X$ and $X$ is given in advance.  For our purposes, we need a slightly weaker version than the one presented in \cite{soberon2016robust} (see also \cite{Clarksonradon}).

\begin{lemma}\label{lemma-halfspace}
Let $X \subset \rr^n$ be a finite set.  Then, there is a family $\mathcal{H}$ of $|X|^n$ half-spaces in $\rr^n$, each containing $0$, such that the following holds.  For every subset $Y \subset X$, we have $ 0 \in \conv Y$ if and only if $Y \cap H \neq \emptyset$ for all $H \in \mathcal{H}$. \qed
\end{lemma}

\begin{proof}[Sketch of proof]
$0$ belongs to $\conv Y$ if and only if there is no hyperplane separating $0$ from $Y$. There are infinitely many candidate hyperplanes, but they can be grouped into equivalence classes according to which subset of $X$ they separate from $0$. We just need one representative from each class. The number of such possible subsets is equal, under duality, to the number of cells into which $|X|$ hyperplanes partition $\rr^n$.
\end{proof}

\subsection{Hoeffding's inequality}

Our main probabilistic tool will be Hoeffding's inequality.

\begin{theorem}[Hoeffding 1963, \cite{Hoe63}]
	Given $n$ independent random variables $x_1, \ldots, x_N$ such that $0 \le x_i \le 1$, let $y=x_1 + \ldots + x_N$. For all $\lambda \ge 0$, we have
	\[
	\mathbb{P}\left[y < \mathbb{E}(y) - \lambda\right] < e^{-2\lambda^2/N}.
	\]
\end{theorem}

The expert reader may know that Hoeffding proved a slightly different inequality: $\mathbb{P}\left[y >\mathbb{E}(y) + \lambda\right] < e^{-2\lambda^2/N}$.  It suffices to apply the inequality to the variables $z_i = 1 - x_i$ to obtain the other bound.  This is a special case of Azuma's inequality (with a slightly different constant in the exponent, which would not change the main result significantly) \cite{Azu67}.  These inequalities carry at their heart the central limit theorem, which is why such an exponential decay is expected in the tails of the distribution.  See \cite{alonspencer} for references on the subject.

\section{Proof of Theorem \ref{theorem-epsilon-caratheodory}}\label{section-proof}
\begin{proof}	

We first prove that $\varepsilon > ((r-1)/r)^m$ is necessary for Theorem \ref{theorem-main}, which also implies the lower bound for Theorem \ref{theorem-epsilon-caratheodory}.  Given $N$ points in $\rr^d$ and $m$ partitions $P_1, \ldots, P_m$, of them, let us find a subset of size greater than or equal to $N((r-1)/r)^m$ in which no $P_k$ induces a Tverberg partition.  First, notice that one of the parts of $P_1$ must have at most $N/r$ points.  If we remove them, then there are at least $N(1-1/r)$ points left.  We can repeat the same argument, and, among the points we have left, one of the parts induced by $P_2$ must have at most a $(1/r)$-fraction of them.  Removing those leaves us with at least $N(1-1/r)^2$ points.  We proceed this way and end up with a set ${Y}$ of at least $N(1-1/r)^m$ points, such that $P_k({Y})$ has at least one empty component for each $k = 1,\ldots, m$.  Therefore, none of these is a Tverberg partition.

Assume now that $\varepsilon > ((r-1)/r)^m$.  We want to prove that there are $m$ transversals as the theorem required.  We choose (with foresight) $A = (Nr)^n$, and $\lambda > \sqrt{m N \ln A}$.  Define a sequence $N_0, N_1, \ldots$ by $N_0 = N$ and $N_k = N_{k-1}(1-1/r) + \lambda$ for $k \ge 1$.  If we apply Lemma \ref{lemma-halfspace} to $\cup \mathcal{F}$, we obtain a family $\mathcal{H}$ of $A$ halfspaces, all containing $0$, which are enough to check if the convex hulls of the transversals we construct contain $0$.

We consider each $F \in \mathcal{F}$ as a color class.  For $k=1, \ldots, m$, we will construct $P_k$ and a family $\gimel_k$ of sets of color classes such that the following properties hold:
\begin{itemize}
	\item given $\mathcal{G} \subset \mathcal{F}$ such that $0 \not\in \conv(P_{k'}(\mathcal{G}))$ for all $k'=1,\ldots,k$, there must be a $\mathcal{V} \in \gimel_k$ such that $\mathcal{G} \subset \mathcal{V}$,
	\item if $\mathcal{V} \in \gimel_k$, then $|\mathcal{V}|\le N_k$, and
	\item $|\gimel_k| \le A^k$.
\end{itemize}
We can consider $\gimel_0 = \{\mathcal{F}\}$.  We construct $P_k$ inductively, assuming $\gimel_{k-1}$ and $P_{k'}$ have been constructed for $k' < k$.  We start by choosing $P_k$ randomly.  For each $F \in \ff$, we pick $y_F^k \in F$ uniformly and independently.  Then, we denote $P_k = \{y^k_F : F \in \ff\}$.

Given a half-space $H \in \mathcal{H}$, consider the random variable

\[
x^k_F(H) = \begin{cases}
1 & \mbox{if } y^k_F \in H \\
0 & \mbox{otherwise }
\end{cases}
\]

Since $0 \in \conv(F)$, we know that $\mathbb{E}(x^k_F(H)) \ge 1/r$.  By linearity of expectation, for each $\mathcal{V} \in \gimel_{k-1}$ we have
\[
\mathbb{E}\left[\sum_{F \in \mathcal{V}}x^k_F(H)\right]  \ge \frac{1}{r}|\mathcal{V}|.
\]
Since all variables $x^k_F(H), x^k_{F'}(H)$ are independent for $F \neq F'$, Hoeffding's inequality gives
\[
\mathbb{P}\left[\sum_{F \in \mathcal{V}}x^k_F(H) < \frac{|\mathcal{V}|}{r}-\lambda\right]< e^{-2\lambda^2/|\mathcal{V}|} \le e^{-2\lambda^2/N}
\]
Therefore the union bound gives
\begin{align*}
\mathbb{P}\left[\exists H \in \mathcal{H} \ \exists \mathcal{V} \in \gimel_{k-1}\mbox{ such that } \sum_{F \in \mathcal{V}}x^k_F(H) < \frac{|\mathcal{V}|}{r}-\lambda\right]  & \le A \cdot |\gimel_{k-1}| \cdot e^{-2\lambda^2/	N} \\
 & \le A^k e^{-2\lambda^2/N} < 1
\end{align*}
by the choice of $\lambda$.

Therefore, there is a choice of $P_k$ such that for all $\mathcal{V} \in \gimel_{k-1}$ and all half-spaces $H \in \mathcal{H}$, we have 
\[
\sum_{F \in \mathcal{V}}x^k_F(H) \ge \frac{|\mathcal{V}|}{r}-\lambda.
\]
We fix $P_k$ to be this choice.  We are ready to construct $\gimel_k$.  For each $\mathcal{V} \in \gimel_{k-1}$ and each half-space $H \in \mathcal{H}$, we construct the set $\mathcal{V}' = \{F \in \mathcal{V} : x^k_F(H) = 0\}$.  We call $\gimel_k$ to the family of all sets that can be formed this way.  Let us prove that $\gimel_k$ satisfies all the desired properties.

\begin{claim}
Given $\mathcal{G} \subset \mathcal{F}$ such that $0 \not\in \conv(P_{k'}(\mathcal{G}))$ for all $k'=1,\ldots,k$, there must be a $\mathcal{V}' \in \gimel_k$ such that $\mathcal{G} \subset \mathcal{V}'$.
\end{claim}

\begin{proof}
If $0 \not\in \conv(P_{k'}(\mathcal{G}))$ for all $k'=1,\ldots,k$, we already know that there must be a $\mathcal{V} \in \gimel_{k-1}$ such that $\mathcal{G} \subset \mathcal{V}$.  Since $0 \not\in \conv(P_k(\mathcal{G}))$, there must be a half-space $H\in \mathcal{H}$ containing $0$ such that $x^k_F(H) = 0$ for all $F \in \mathcal{G}$.  Therefore, there is a $\mathcal{V}' \in \gimel_k$ with $\mathcal{G} \subset \mathcal{V'}$.
\end{proof}

\begin{claim}
If $\mathcal{V'} \in \gimel_k$, then $|\mathcal{V'}|\le N_k$.
\end{claim}

\begin{proof}
Let $\mathcal{V} \in \gimel_{k-1}$, $H\in \mathcal{H}$ be the family and half-space that defined $\mathcal{V'}$, respectively.  Then,
\[
|\mathcal{V}'| = \sum_{F \in \mathcal{V}}(1-x^k_F(H)) \le |\mathcal{V}|\left(1-\frac{1}{r}\right) + \lambda \le N_{k-1}\left(1-\frac{1}{r}\right)+ \lambda = N_k.
\]
\end{proof}

\begin{claim}
We have $|\gimel_k| \le A^k$.
\end{claim}

\begin{proof}
By construction, $|\gimel_k| \le |\gimel_{k-1}| \cdot A \le A^k$.
\end{proof}

This concludes the construction of $P_1, \ldots, P_m$.  

If $\mathcal{G} \subset \mathcal{F}$ is such that $0 \not \in \conv P_k (\mathcal{G})$ for $k=1, \ldots, m$, then there must be a $\mathcal{V} \in \gimel_m$ such that $\mathcal{G} \subset \mathcal{V}$.

Recall that $m$ was chosen so that $((r-1)/r)^m < \varepsilon$.  Therefore
\[
|\mathcal{G}| \le |\mathcal{V}| \le N_m \le N\left(\frac{r-1}{r}\right)^m + r\lambda < \varepsilon N,
\]
where the last inequality holds if $N$ is large enough, as $\lambda = O(\sqrt{N\ln N})$.
\end{proof}
\section{A Colorful version}\label{section-colored}

Another important variation of Tverberg's theorem is the following conjecture by B\'ar\'any and Larman.

\begin{conjecture}[Colorful Tverberg; B\'ar\'any, Larman 1992 \cite{Barany:1992tx}]
	For any given $d+1$ sets $F_1, \ldots, F_{d+1}$ of $r$ points each in $\rr^d$, there is a Tverberg partition $X_1, \ldots, X_r$ of their union such that for all $i,j$ we have $|F_i \cap X_j|=1$.
\end{conjecture}

A partition $X_1, \ldots, X_r$ with $|F_i \cap X_j|=1$ for all $i,j$ is called a \textit{colorful partition}.  Consult \cite{Blagojevic:2014js, Blagojevic:2011vh, blago15} and the references therein the current solved cases and techniques.  We present an $\varepsilon$-version of the conjecture above in the following theorem.  Let $p_r\sim 1 - 1/e$ be the probability that a random permutation of a set with $r$ elements has fixed points.

\begin{theorem}\label{theorem-tverbergcolored}
	Let $r,d$ be positive integers and $\varepsilon >0$ be a real number.  There is an $m_{\col} = m_{\col}(\varepsilon, r)$ such that the following holds.  For a sufficiently large $N$, if we are given $N$ sets $F_1, \ldots, F_N$ of $r$ points in $\rr^d$ each, then there are $m_{\col}$ colorful partitions of $\mathcal{F} = \{F_1, \ldots, F_N\}$ such that for any $\mathcal{G} \subset \mathcal{F}$ with $|\mathcal{G}| \ge \varepsilon |\mathcal{F}|$, at least one of the partitions induces a colorful Tverberg partition on $|\mathcal{G}|$.  Moreover, we have
	\[
	m_{\col} \le \left\lfloor \frac{\ln\left({\varepsilon} \right)}{\ln (1-p_r)} \right\rfloor +1
	\]
\end{theorem}

We should note that the theorem above gives $m \sim 1 + \ln (1/\varepsilon)$ if $r$ is large enough.  This is related to the colorful version from \cite{soberon2016robust}, which seeks the smallest $\varepsilon$ for which $m_{\col}(\varepsilon, r) = 1$.  Using our notation, the main conjecture in that paper states the following.

\begin{conjecture}
For all $\varepsilon >0$ and any positive integer $r$, we have
\[
m_{\col}(\varepsilon, r) = 1.
\]
\end{conjecture}

To prove Theorem \ref{theorem-tverbergcolored}, we also use Sarkaria's transformation.  In order to translate the conditions on the colors through the tensor products, we need the following definition.

A set $B$ is an $r$-block if it is an $r \times r$ array of points in $\rr^n$ such that the convex hull of each column contains the origin.  A \textit{colorful transversal} of an $r$-block $B$ is a subset of $r$ points of $B$ that has exactly one point of each column and exactly one point of each row.  Given a family $\mathcal{B}$ of $r$-blocks, a \textit{colorful transversal} for $\mathcal{B}$ is the result of putting together a colorful transversal for each block.  If we apply Sarkaria's technique, colorful partitions in $\rr^d$ become colorful transversals of $r$-blocks in $\rr^n$. Theorem \ref{theorem-tverbergcolored} is then implied by the following.

\begin{theorem}\label{theorem-coloredcarth-epsilon}
	Let $n,r$ be positive integers and $\varepsilon >0$ be a real number.  $m_{\col} = m_{\col}(\varepsilon, r)$ such that the following holds.  For a sufficiently large $N$, if we are given $N$ $r$-blocks $B_1, \ldots, B_N$ in $\rr^n$, there are $m_{\col}$ colorful transversals $P_1, \ldots, P_m$ of $\mathcal{B} = \{B_1, \ldots, B_N\}$ such that for any $\mathcal{G} \subset \mathcal{B}$ with $|\mathcal{G}| > \varepsilon |\mathcal{B}|$ for at least one $k$ we have $0 \in \conv (P_k (\mathcal{G}))$.  Moreover, we have
	\[
	m_{\col} \le \left\lfloor \frac{\ln\left({\varepsilon} \right)}{\ln (1-p_r)} \right\rfloor +1.
	\]
\end{theorem}

We also need the observation from \cite{soberon2016robust} that, for any $r$-block and any half-space $H$ that contains the origin, the probability that a random colorful transversal has points in $H$ is greater than or equal to $p_r$.

\begin{proof}
We proceed in a similar fashion to the proof of Theorem \ref{theorem-epsilon-caratheodory}.

Assume that $\varepsilon > (1-p_r)^m$.  We want to prove that there are $m$ transversals as the theorem requires.  We choose (with foresight) $A = (Nr^2)^n$, and $\lambda > \sqrt{mN\ln A}$.  Define a sequence recursively by $N_0 = N$ and $N_k = N_{k-1}(1-p_r) + \lambda$.   If we apply Lemma \ref{lemma-halfspace} to $\cup \mathcal{B}$, we obtain a family $\mathcal{H}$ of $A$ half-spaces, all containing $0$, which are enough to check if the convex hulls of the colorful transversals we construct contain $0$.

For $k = 1, \ldots, m$, we will construct $P_k$ and a family $\gimel_k$ of sets of $r$-blocks with the following properties.

\begin{itemize}
	\item Given $\mathcal{G} \subset \mathcal{B}$ such that $0 \not\in \conv(P_{k'}(\mathcal{G}))$ for all $k'=1,\ldots, k$, there must be a $\mathcal{V} \in \gimel_k$ such that $\mathcal{G} \subset \mathcal{V}$,
	\item if $\mathcal{V} \in \gimel_k$, then $|\mathcal{V}| \le N_k$, and
	\item $|\gimel_k| \le A^k$.
\end{itemize}
  We can consider $\gimel_0 = \{\mathcal{B}\}$ to start the induction.  We construct $P_k$ inductively, assuming $\gimel_{k-1}$ and $P_{k'}$ have been constructed for $k' < k$.  We first choose $P_k$ randomly.  For each $B \in \mathcal{B}$, we pick a colorful transversal $y^k_B$ randomly and independently.  Then, we denote $P_k = \{y^k_B : B \in \mathcal{B}\}$.
  
  Given a half-space $H \in \mathcal{H}$, consider the random variable
  \[
  x^k_B(H) = \begin{cases}
  	1 & \mbox{if } y^k_B \cap H \neq \emptyset \\
  	0 & \mbox{otherwise.}
  \end{cases}
  \]
  
  Since $\mathbb{E} [x^k_B(H)]\ge p_r$ for each $B \in \mathcal{B}, H \in \mathcal{H}$, we have that for any $\mathcal{V} \in \gimel_{k-1}$
  \[
  \mathbb{E} \left[\sum_{B \in \mathcal{V}}x^k_B(H)\right]\ge |\mathcal{V}|p_r
  \]

 Since all variables $x^k_B(H)$, $x^k_{B'}(H)$ are independent for $B \neq B'$, Hoeffidng's inequality gives
 
 \[
 \mathbb{P}\left[\sum_{B \in \mathcal{V}}x^k_B(H) < |\mathcal{V}|p_r - \lambda \right] < e^{-2\lambda^2/|\mathcal{V}|} \le e^{-2\lambda^2/N}.
 \]
 Therefore
 
 \begin{align*}
 	\mathbb{P}\left[\exists H \in \mathcal{H} \exists \mathcal{V} \in \gimel_{k-1} \sum_{B \in \mathcal{V}}x^k_B(H) < |\mathcal{V}|p_r - \lambda \right] & < A \cdot |\gimel_{k-1}| e^{-2\lambda^2/N} \\
 	 & \le A^k e^{-2\lambda^2/N} < 1
 \end{align*}
 
by the choice of $\lambda$.

Therefore, there must be a choice of $P_k$ such that for all $H \in \mathcal{H}$ and all $\mathcal{V} \in \gimel_{k-1}$ we have
\[
\sum_{B \in \mathcal{V}}x^k_B(H) \ge |\mathcal{V}|p_r - \lambda.
\]
We fix $P_k$ to be this choice.  In order to form $\gimel_k$, for each $\mathcal{V} \in \gimel_{k-1}$ and $H \in \mathcal{H}$, we include the set $\{B \in \mathcal{V}: x^k_B(H) = 0\}$.  Proving that $\gimel_k$ satisfies the desired properties and that this implies the conclusion of Theorem \ref{theorem-coloredcarth-epsilon} follows from arguments analogous to those at the end of section \ref{section-proof}.
 \end{proof}

\section{Acknowledgments}

The author would like to thank the careful comments of two anonymous referees, which have significantly improved the quality of this paper.
\newcommand{\etalchar}[1]{$^{#1}$}
\providecommand{\bysame}{\leavevmode\hbox to3em{\hrulefill}\thinspace}
\providecommand{\MR}{\relax\ifhmode\unskip\space\fi MR }
\providecommand{\MRhref}[2]{%
  \href{http://www.ams.org/mathscinet-getitem?mr=#1}{#2}
}
\providecommand{\href}[2]{#2}

\noindent Pablo Sober\'on \\
\textsc{
Mathematics Department \\
Northeastern University \\
Boston, MA 02445
}\\[0.1cm]

\noindent \textit{E-mail address: }\texttt{pablo.soberon@ciencias.unam.mx}

\end{document}